\documentclass[11pt]{article}

\usepackage{amsmath,amsthm,verbatim,amssymb,amsfonts,amscd, graphicx}
\usepackage{graphics}
\usepackage{tikz-cd}
\theoremstyle{plain}
\newtheorem{theorem}{Theorem}
\newtheorem{corollary}{Corollary}
\newtheorem{lemma}{Lemma}

\theoremstyle{definition}

\DeclareMathOperator{\Arf}{Arf}
\DeclareMathOperator{\Spin}{Spin}
\DeclareMathOperator{\Pin}{Pin}
\DeclareMathOperator{\Hom}{Hom}
\DeclareMathOperator{\Sp}{Sp}

\newcommand{\R}{\mathbb{R}}

\newcommand{\Q}{\mathbb{Q}}
\newcommand{\Z}{\mathbb{Z}}
\newcommand{\N}{\mathbb{N}}

\begin{document}

\title{Some properties of $\Pin^\pm$-structures on compact surfaces}
\author{Michael R. Klug and Luuk Stehouwer}
\maketitle

\begin{abstract}
    We show that two $\Pin$-structures on a surface differ by a diffeomorphism of the surface if and only if they are cobordant (for comparison, the analogous fact has already been shown for $\Spin$-structures). We give a construction that shows that this does not extend to dimensions greater than two. In addition, we count the number of $\Pin$-structures on a surface in a given cobordism class.   
\end{abstract} 

This note discusses $\Pin$-structures on closed 2-dimensional manifolds (for simplicity, here we mean $\Pin^-$-structures, though $\Pin^+$-structures are also considered) and how certain results concerning $\Spin$-structures extend to $\Pin$-structures.  $\Spin$-structures on a fixed compact surfaces are classified up to diffeomorphism and up to cobordism by the Arf invariant, and further, the exact number of $\Spin$-structures with a given Arf invariant is known (see section \ref{sec:spin}).  For $\Pin$ structures, the cobordism class of the $\Pin$-structure is determined by the Brown invariant which we show also determines the diffeomorphism class of the $\Pin$-structure (see section \ref{sec:pin} for the relevant definitions).  We also determine the number of $\Pin$-structures on a given surface with a given Brown invariant.  To conclude, we discuss how these results fail to extend to dimensions greater than two.   

In a physical setting (which is not necessary for a reading of this paper, but which perhaps provides context), $\Pin$-structures are necessary to define spinors on unorientable Euclidean spacetimes, for example in the presence of time-reversing symmetries. 
Furthermore, their bordism groups have recently caught a lot of attention, because unitary invertible topological field theories are classified by their partition function, which is a homomorphism from the bordism group to $U(1)$.
Such invertible field theories have applications in anomalies of quantum field theories and symmetry-protected topological phases of matter.
On surfaces in particular, homomorphisms from the $\Pin^{\pm}$-bordism groups to $U(1)$ classify symmetry-protected topological phases protected by a time-reversing symmetry $T$ with $T^2 = \mp$.
In the case of $\Omega^{\Pin^-}_2$ the low-energy limit of the Kitaev chain is a generator of $\Hom(\Omega^{\Pin^-}_2,U(1)) \cong \Z/8\Z$, as shown in \cite{debray2018arf}.

In section \ref{sec:spin}, we provide context for our work on $\Pin$-structures by recalling the analogous results for $\Spin$-structures.  In section \ref{sec:pin}, we present our main results, echoing the results for $\Spin$-structures from section \ref{sec:spin}.  In section \ref{sec:high_dim}, we show the distinction between being diffeomorphic and being cobordant for $\Pin$-structures in dimensions greater than two. 

\subsection*{Acknowledgements}

Both authors would like to thank the
Max Planck Institute for Mathematics in Bonn where this research was carried out for support.  In addition, we want to thank Jean Raimbault who told us the proof of Lemma \ref{lem:higher} and pointed us to \cite{hyperbolicmanifolds}. 

\section{Introduction and $\Spin$-structures on closed surfaces} \label{sec:spin}

In this section, we collect several results concerning $\Spin$-structures on compact surfaces in order to set the stage for the following section where we discuss analogous results for $\Pin^{\pm}$-structures.  

Let $\Sigma$ be a closed oriented genus $g$ surface. A map
 $$
 q : H_1(\Sigma; \mathbb{Z}/2\mathbb{Z}) \to \mathbb{Z}/2\mathbb{Z}
$$
which satisfies
$$
 q(x+y) = q(x) + q(y) + x \cdot y
$$
 for all $x,y \in H_1(\Sigma, \mathbb{Z}/2\mathbb{Z})$ where $x \cdot y$ denotes the intersection of the classes $x$ and $y$ is called a \emph{quadratic refinement} of the intersection form.  

 The \emph{Arf invariant} of such a quadratic refinement $q$, denoted $\Arf(q)$, can be defined as the value in $\mathbb{Z}/2\mathbb{Z}$ where $q$ has a larger preimage (see for example Corollary 9.5 of \cite{saveliev_book}). 
 Alternatively, $\Arf(q)$ is   
 $$
 \Arf(q) = \sum_{i=1}^{g} q(a_i) q(b_i) \in \mathbb{Z}/2\mathbb{Z}
 $$
 where $a_1,b_1,...,a_g,b_g$ is a symplectic basis for $H_1(\Sigma, \mathbb{Z}/2\mathbb{Z})$ with respect to the intersection form. 
 This is independent of the choice of symplectic basis \cite{arf}.

Recall that a spin structure on an oriented surface is a $\Spin_2$-structure $s = (P,\alpha)$ on the tangent bundle, where $P \to \Sigma$ is a principal $\Spin_2$-bundle and $\alpha: P \times_{\Spin_2} \R^2 \to T \Sigma$ an isomorphism of vector bundles. 
We only consider spin structures up to homotopy, so an \emph{isomorphism of spin structures} between $(P_1,\alpha_1)$ and $(P_2,\alpha_2)$ will be a map $f: P_1 \to P_2$ of $\Spin_2$-bundles (covering the identiy on $\Sigma$) such that $\alpha_2 \circ f_*$ is homotopic to $\alpha_1$.
We will identify isomorphic spin structures from now.
In \cite{johnson_quadratic}, Johnson gives a bijection between the set of isomorphism classes of spin structures on $\Sigma$ and the set of quadratic refinements of the intersection form.  The \emph{Arf invariant} of a spin structure $s$ on $\Sigma$ is defined to be Arf invariant of the corresponding quadratic refinement $q_s$ and denoted $\Arf(s)$.

 A \emph{$\Spin$-diffeomorphism} is a coarser notion of equivalence for two spin structures on $\Sigma$ in which we also allow changes of $\Sigma$ by a diffeomorphism.
 Two spin structures $s_1$ and $s_2$ on $\Sigma$ are called \emph{$\Spin$-diffeomorphic} if there is a diffeomorphism $\phi : \Sigma \to \Sigma$ with the property that $\phi^\ast s_1 = s_2$.  
 By considering the effect of a diffeomorphism $\Sigma \to \Sigma$ on $H_1(\Sigma; \mathbb{Z})$, there is a homomorphism from the mapping class group of $\Sigma$ to the group $\Sp_{2g}(\mathbb{Z})$ of integral matrices that preserve the intersection form on $H_1(\Sigma, \mathbb{Z})$ (we have implicitly chosen a basis for $H_1(\Sigma;\mathbb{Z})$ for notational simplicity, but this is not necessary).  
 We thus obtain an action of $\Sp_{2g}(\mathbb{Z})$ on the set of spin structures on $\Sigma$.  
 In terms of the quadratic refinements, this action is given by $A \cdot q(x) = q(A^{-1} x)$ for $q$ a quadratic refinement, $A \in \Sp_{2g}(\mathbb{Z})$, and $x \in H_1(\Sigma; \mathbb{Z})$.  Therefore, $q$ and $A \cdot q$ have the same number of $0$ and $1$ values and therefore have the same Arf invariant.  
 This homomorphism to $\Sp_{2g}(\mathbb{Z})$ is a surjection (see for example Theorem 6.4 of \cite{farb_margalit}) and from Corollary 2 of \cite{johnson_quadratic}, it follows that two spin structures on $\Sigma$ are $\Spin$-diffeomorphic if and only if they have the same Arf invariants.  In \cite{johnson_bc}, Johnson proves that there are $2^{g-1}(2^g + 1)$ spin structures on $\Sigma$ with Arf invariant $0$ and $2^{g-1}(2^g - 1)$ spin structures on $\Sigma$ with Arf invariant $1$.  

 Given two closed oriented surfaces $\Sigma_1$ and $\Sigma_2$, with respective spin structures $s_1$ and $s_2$, we can consider the relation of spin cobordance between them (see for example \cite{kirbytaylor}).  
 In \cite{kirbytaylor}, it is shown that $(\Sigma_1, s_1)$ and $(\Sigma_2, s_2)$ are spin cobordant if and only if $\Arf(s_1) = \Arf(s_2)$.  In other words, the map
 \begin{align*}
	 \Arf : \Omega_2^{\Spin} &\to \mathbb{Z}/2\mathbb{Z} \\
		 (\Sigma, s)   &\mapsto \Arf(s)
 \end{align*}
is an isomorphism.  From the proceeding remarks on $\Spin$-diffeomorphism, we thus conclude that two spin structures on $\Sigma$ are cobordant if and only if they are $\Spin$-diffeomorphic.

\section{$\Pin^\pm$ structures on closed surfaces} \label{sec:pin}

We now mimic the discussion of the previous section for the case of $\Pin^{\pm}$-structures.
In particular, we can consider when two $\Pin^{\pm}$-structures are $\Pin$-diffeomorphic or cobordant.
We show that these two equivalence relations are equal for closed surfaces, just as we saw for $\Spin$-structures in the previous section.  We also discuss how many $\Pin^{\pm}$-structures on a given surface are in a given cobordism class.  Most of our discussion concentrates on the case of $\Pin^{-}$-structures; at the end of this section we give pointers to the literature where the relevant results concerning $\Pin^{+}$-structures can be found.  

There exist two nonisomorphic double covers $\Pin^{\pm}_2 \to O_2$ that restrict to the double cover $\Spin_2 \to SO_2$.  
Reflections in $O_2$ lift to elements with square $1$ in $\Pin^+_2$ but to elements with square $-1$ in $\Pin^-_2$, which distinguishes the two double covers.
A $\Pin^{\pm}$-structure $p$ on a (now not necessarily orientable) closed surface $F^2$ is a $\Pin^{\pm}_2$-structure on the tangent bundle, which we again identify up to isomorphism.
For more background on pin structures see \cite{kirbytaylor}.  A map
 $$
 e : H_1(\Sigma; \mathbb{Z}/2\mathbb{Z}) \to \mathbb{Z}/4\mathbb{Z}
$$
which satisfies 
$$
 e(x+y) = e(x) + e(y) + 2 \cdot (x \cdot y)
$$
 for all $x,y \in H_1(\Sigma, \mathbb{Z}/2\mathbb{Z})$ is called a \emph{quadratic enhancement} of the intersection form.  
 Here $x \cdot y$ denotes the intersection of the classes $x$ and $y$ as an element of $\Z/2\Z$ and the $2 \cdot$ indicates the inclusion $\mathbb{Z}/2\mathbb{Z} \to \mathbb{Z}/4\mathbb{Z}$ .

 Let $V$ be a finite-dimensional $\mathbb{Z}/2\mathbb{Z}$-vector space with a symmetric nondegenerate bilinear form 
$$
 \cdot : V \otimes V \to \mathbb{Z}/2\mathbb{Z}
$$ 
and a map 
$$
 e : V \to \mathbb{Z}/4\mathbb{Z}
$$
which satisfies
$$
 e(x+y) = e(x) + e(y) + 2 \cdot (x \cdot y)
$$
for all $x,y \in V$.  Then there is an $\mathbb{Z}/8\mathbb{Z}$-valued invariant, denoted $\beta(e)$, called the Brown invariant which is defined as in Figure \ref{fig:compass} - see also section 3.2 of \cite{kirby_melvin}.  One property of the Brown invariant that will be important to our discussion is that it is additive in the sense that if we are given a pair $(V, \cdot, e_1)$ and $(W, \cdot, e_2)$ then we can form there sum in a natural way and 
$$
\beta(e_1 + e_2) = \beta(e_1) + \beta(e_2)
$$

\begin{figure}  	
	\centering
	\includegraphics[width=8cm]{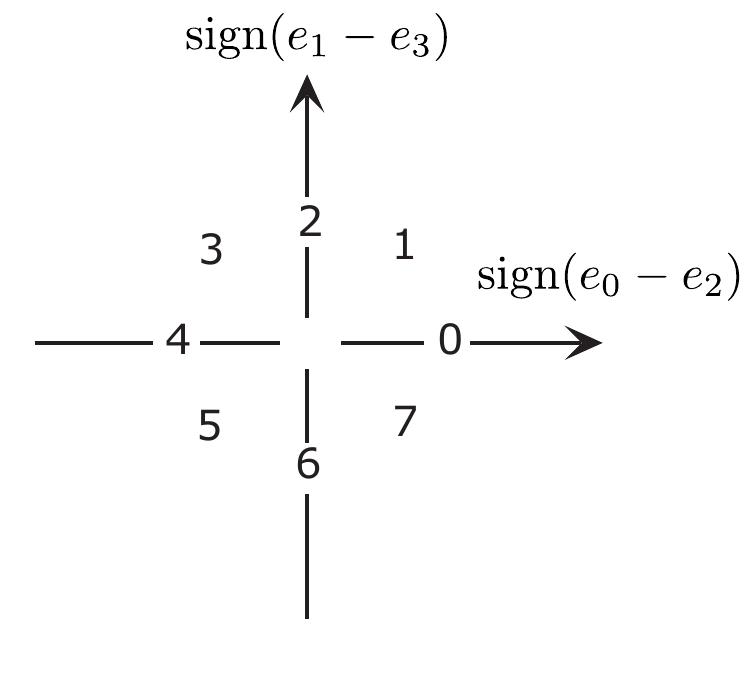}
	\caption{This figure gives the definition of the Brown invariant.  The quantities $e_i$ are the number of times that the enhancement $e$ takes on the value $i \in \mathbb{Z}/4\mathbb{Z}$ and $\text{sign}(e_0-e_2)$ and $\text{sign}(e_1-e_3)$ denote the signs (positive, zero, or negative) of the respective integers.  The quadrant is then determined by these two signs and the result is the Brown invariant.  For example, if $e_0 = e_2$ and $e_1 > e_3$ then the Brown invariant is $2$.}  
	\label{fig:compass}
\end{figure}

In \cite{kirbytaylor}, Kirby and Taylor give a construction, similar to the construction given by Johnson mentioned in the previous section, that gives a bijection between the isomorphism classes of $\Pin^-$ structures $p$ on $F$ and the set of quadratic enhancements of the intersection form of $F$.  Kirby and Taylor \cite{kirbytaylor} give an isomorphism
 \begin{align*}
	 \beta : \Omega_2^{\Pin^{-}} &\to \mathbb{Z}/8\mathbb{Z} \\
		 (F, p)   &\mapsto \beta(p)
 \end{align*}
where $\beta(p)$ is defined as the Brown invariant $\beta(e_p)$ where $e_p$ is the quadratic enhancement corresponding to the $\Pin^+$ structure $p$.

\begin{theorem} \label{thm:main}
Let $F$ be a closed surface.  Two $\Pin^-$-structures on $F$ are $\Pin$-diffeomorphic if and only if they have the same Brown invariant (i.e., if and only if they are cobordant).  
\end{theorem}

\begin{proof}
	Let $p_1$ and $p_2$ be two $\Pin^-$-structures on $F$ and let $e_1$ and $e_2$ be the respective associated quadratic enhancements.  Suppose there is a diffeomorphism $\phi : F \to F$ with $\phi^\ast p_1 = p_2$.  By the construction of the quadratic enhancement from the $\Pin^-$-structure (see \cite{kirbytaylor}), $e_1(\phi_\ast x) = e_2(x)$ and therefore $\beta(e_1) = \beta(e_2)$, thus $\beta(p_1) = \beta(p_2)$.  

	Conversely, suppose that $\beta(p_1) = \beta(p_2)$.  If $F$ is orientable, then there is an orientation of $F$ and two spin structures $s_1$ and $s_2$ of $F$ with that orientation such that $s_1$ induces $p_1$ and $s_2$ induces $p_2$ (see \cite{kirbytaylor}).  Since $4 \Arf(s_i) = \beta(p_i)$ (where $4$ denotes the inclusion $\mathbb{Z}/2\mathbb{Z} \to \mathbb{Z}/8\mathbb{Z})$), it follows that  $\Arf(s_1) = \Arf(s_2)$ and hence, by the result of Johnson mentioned in the introduction, there exists a diffeomorphism $\phi : F \to F$ such that $\phi^\ast s_1 = s_2$.  Therefore $\phi^\ast p_1 = p_2$, as desired.  

	Assume that $F$ is nonorientable.  We proceed by induction on the nonorientable genus of $F$, defined to be $1-\chi(F)$.  In the genus $1$ case of $\mathbb{R}P^2$, the two $\Pin^-$-structures on $\mathbb{R}P^2$ have distinct Brown invariants (namely $\pm 1$) so there is nothing to check.  Note that, by the construction of a quadratic enhancement from a $\Pin^-$-structure, we know that if $e$ is such a quadratic enhancement and $x \in H_1(F; \mathbb{Z}/2\mathbb{Z})$ with $w_1(x) = 1$, then $e(x) \in \{1,3\}$.  
	For the higher genus case, choose a diffeomorphism of $F$ with the connected sum of $k$ copies of $\mathbb{R}P^2$.  We then have a basis for $H_1(F;\mathbb{Z}/2\mathbb{Z})$ given by the $\mathbb{R}P^1$ circles in each $\mathbb{R}P^2$-summand, call these classes $x_1,...,x_k$.  
	
	First, assume that $e_1(x_i) = e_2(x_j)$ for some indices $i$ and $j$.  
	Then, $p_1$ and $p_2$ are each $\Pin$-diffeomorphic to new $\Pin^-$-structures (which we again call $p_1$ and $p_2$) such that for the respective induced quadratic enhancements (which we again call $e_1$ and $e_2$) we have $e_1(x_1) = e_2(x_1)$.  
	These respective diffeomorphisms are the diffeomorphisms that interchange the 1st and $i$th (respectively 1st and $j$th) $\mathbb{R}P^2$-summands.  
	Let $F'$ denote the subsurface of $F$ containing the $(k-1)$ $\mathbb{R}P^2$-summands that exclude the summand with $x_1$.  Note that there are two $\Pin^-$-structures on $S^1$. 
	The unique $\Pin^-$-structure on the disk $D^2$, when restricted to the boundary $S^1$, is the same as the $\Pin^-$-structure coming from restricting the $\Pin^-$-structure on $F'$ to the boundary.  
	Therefore, we can cap off the surface $F'$  using a disk to give a surface which we denote by $\overline{F}'$, and both of $p_1$ and $p_2$ will extend over this disk uniquely to give $\Pin^-$-structures $\overline{p_1}$ and $\overline{p_2}$ on $\overline{F}'$.  
	Now, by the additivity of the Brown invariant with respect to disjoint union, since $\beta(p_1) = \beta(p_2)$, we have that $\beta(\overline{p_1}) = \beta(\overline{p_2})$.  
	Therefore, by the induction hypothesis, there exists a diffeomorphism $\psi : \overline{F}' \to \overline{F}'$ such that $\psi^\ast \overline{p_1} = \overline{p_2}$. 
	By isotoping $\psi$, we may assume that $\psi$ is the identity on the disk that was used for the capping off, and therefore, by letting $\phi: F \to F$ be $\psi$ on the portion of $F$ that is $\overline{F}'$ without the capping off disk, and the identity on the last $\mathbb{R}P^2$-summand, we have $\phi^\ast p_1 = p_2$, as desired.

	Finally, we consider the case where $e_1(x_i) =1$ and $e_2(x_i) = 3$ for all $1 \leq i \leq k$.  In this case, by the additivity of the Brown invariant, we have $k = -k \pmod{4}$ and in particular, $k \geq 3$.  Consider the embedded curve in $F$ obtained by banding $x_1$ to $x_2$ and then banding the result to $x_3$.  Call this curve $y$, then 
	$$
	e_1(y) = 1 + 1 + 1 = 3
	$$
	and $w_1(y) = 1$.  By applying a diffeomorphism $F \to F$ that interchanges $x_1$ and $y$, we see that $p_1$ is $\Pin$-diffeomorphic to a $\Pin^-$-structure on $F$, which we again call $p_1$, with the property that $e_1(x_1) = e_2(x_1)$, where $e_1$ is the new associated quadratic enhancement.  Now applying the argument in the previous paragraph implies that $p_1$ and $p_2$ are $\Pin$-diffeomorphic, as desired.  
\end{proof}

\begin{corollary} \label{cor}
Let $F$ be a closed surface.  Two $\Pin^-$-structures on $F$ are cobordant if and only if they are $\Pin$-diffeomorphic.
\end{corollary}

\begin{proof}
	This follows from Theorem \ref{thm:main} together with the isomorphism $\beta : \Omega_2^{\Pin^{-}} \to \mathbb{Z}/8\mathbb{Z}$.  
\end{proof}

For a closed orientable surface $F$, let $\beta_i(F)$ denote the number of isomorphism classes of $\Pin^-$-structures on $F$ that have Brown invariant equal to $i \in \mathbb{Z}/8\mathbb{Z}$.

\begin{theorem}
Let $\Sigma_g$ be a closed orientable surface of genus $g$, then
$$
	\beta_i(\Sigma_g) = 
	\begin{cases} 
		2^{g-1}(2^g + 1) & \text{if } i = 0 \\
		2^{g-1}(2^g - 1) & \text{if } i = 4 \\
		0 & \text{otherwise} 
   	\end{cases}
$$

Let $N_k$ denote a nonorientable closed surface of nonorientable genus $k$, then for $k$ odd we have
$$
	\beta_i(N_k) = 
	\begin{cases} 
		2^{k-2} + 2^{\frac{k-3}{2}} & \text{if } i = 1 \text{ or } i = 7\\
		2^{k-2} - 2^{\frac{k-3}{2}} & \text{if } i = 3  \text{ or } i = 5\\
		0 & \text{if $i$ is even} 
   	\end{cases}
$$
and for $k$ even
$$
	\beta_i(N_k) = 
	\begin{cases} 
		2^{\frac{3k-6}{2}} & \text{if } i = 0\\
		2^{k-2} & \text{if } i = 2  \text{ or } i = 6\\
		2^{k-2} - 2^{\frac{k-2}{2}} & \text{if } i = 4\\
		0 & \text{if $i$ is even} 
   	\end{cases}
$$
\end{theorem}

\begin{proof}
	The orientable case follows from the fact that on such an orientable surface, all of the $\Pin^-$-structures are induced by $\Spin$-structures, the relationship between $\beta$ and $\Arf$, and Johnson's determination of the distribution of the $\Arf$ invariant for the different $\Spin$-structures on an orientable surface.  Noting that by the construction of a quadratic enhancement from a $\Pin^-$-structure, if $e$ is such a quadratic enhancement and $x \in H_1(F; \mathbb{Z}/2\mathbb{Z})$ with $w_1(x) = 1$, then $e(x) \in \{1,3\}$.  Decompose $N_k$ as the connected sum of $k$-copies of $\mathbb{R}P^2$ and let $x_1,...,x_k \in H_1(N_k; \mathbb{Z}/2\mathbb{Z})$ denote the core circles of the $\mathbb{R}P^1$ in each summand.  Note that there are $|H^1(N_k; \mathbb{Z}/2\mathbb{Z})| = 2^k$ distinct $\Pin^-$-structures on $N_k$ and that by using the identity
	$$
	e(x+y) = e(x) + e(y) + 2 \cdot(x,y)
	$$
together with the fact that $x_1,..,x_k$ is a basis for $ H_1(N_k; \mathbb{Z}/2\mathbb{Z})$, we see that each $\Pin^-$-structure is completely determined by its values on $x_1,...,x_k$, and that each possible assignment of each $x_i$ to either $1$ or $3$ is a valid $\Pin^-$-structure.  

	Restricting attention to the first $\mathbb{R}P^2$-summand, we have either $e(x_1) =1$ or $e(x) = 3$.  Remove the first $\mathbb{R}P^2$-summand, capping both resulting surfaces off with disks, and extend the $\Pin^-$-structures over these disks.  By the additivity of the Brown invariant, the Brown invariant of $F$ is equal to the sum of the Brown invariants of these two surfaces and this Brown invariant for $\mathbb{R}P^2$ is either $1$ or $-1$.  We thus obtain the recursion
	$$
	\beta_i(N_k) = \beta_{i-1}(N_{k-1}) + \beta_{i+1}(N_{k-1})
	$$
Solving explicitly for this recursion yields the result.  
\end{proof}

We now briefly discuss $\Pin^+$-structures.  The analogue of quadratic enhancements $e$ in the case of $\Pin^+$-structures is a map
$$
q : H_1(F; \mathbb{Z}/4\mathbb{Z}) \to \mathbb{Z}/2\mathbb{Z}
$$
satisfying
$$
q(x + y) = q(x) + q(y) + x \cdot y
$$
for all $x,y \in H_1(F; \mathbb{Z}/4\mathbb{Z})$.  In \cite{degtyarev_finashin}, Degtyarev and Finashin consider such maps $q$ and prove that they are in bijective correspondence with $\Pin^+$-structures on $F$.  Additionally, they prove the exact analogue of Corollary \ref{cor} in the case of $\Pin^+$ structures.   Kirby and Taylor compute $\Omega_2^{\Pin^{+}} \cong \mathbb{Z}/2\mathbb{Z}$ and show that any orientable surface with a $\Pin^+$-structure is trivial in $\Omega_2^{\Pin^{+}}$ (see Proposition 3.9 of \cite{kirbytaylor}).  Recall that $\Pin^+$-structures only exist on closed nonorientable surfaces that are a connect sum of an even number of $\mathbb{R}P^2$'s.  Kirby and Taylor show that exactly half of the $\Pin^+$-structures on the Klein bottle are nontrivial in $\Omega_2^{\Pin^{+}}$ (see Proposition 3.9 of \cite{kirbytaylor}) from which it follows that for any closed nonorientable surface, exactly half of the $\Pin^+$-structures are nontrivial in $\Omega_2^{\Pin^{+}}$.

\section{Higher-dimensional examples} \label{sec:high_dim}

In this section, we demonstrate the failure of the analogue of Corollary \ref{cor} in dimensions greater than 2.  We begin with some preliminary lemmas.  Lemma \ref{lem:higher} was provided to us by Jean Raimbault.


\begin{lemma}
\label{lemma}
Let $\Delta \subseteq D$ be a finite index subgroup and consider the short exact sequence
\[
1 \to \Delta \to D \to D/\Delta \to 1.
\]
Then $D/\Delta$ acts on the group cohomology $H^1(\Delta;\Q)$ via 
\[
([d] \cdot \psi)(\delta) = \psi(d \delta d^{-1}) \quad d \in D, \delta \in \Delta.
\]
Moreover, the restriction map 
\[
r: H^1(D;\Q) \to H^1(\Delta;\Q)
\]
induces a bijection onto the fixed points of this action.
In particular, $b_1(\Delta,\Q) \geq b_1(D,\Q)$
\end{lemma}
\begin{proof}
    We first show that the action is well-defined.
    If $\delta,\delta' \in \Delta$, then for a homomorphism $\psi: \Delta \to \Q$ we compute
    \[
    \psi(\delta \delta' \delta^{-1}) = \psi(\delta) + \psi(\delta') - \psi(\delta) = \psi(\delta'),
    \]
    so acting with $d := \delta \in \Delta$ fixes $\psi$.
    Hence the action is independent of the choice of representative.
    
    To show that $r$ maps to the fixed point set, let $\phi: D \to \Q$ be given.
    For $\delta \in \Delta$ and $d \in D$ we can compute
    \[
    \phi(d \delta d^{-1}) = \phi(d) + \phi(\delta) - \phi(d) = \phi(\delta),
    \]
    so $r(\phi)$ is fixed under the aciton of $[d]$.
    
    For injectivity, let $\phi: D \to \Q$ be a homomorphism that restricts to zero on $\Delta$.
    Let $d \in D$ and pick $n \in \N$ such that $d^n \in \Delta$.
    Then 
    \[
    \phi(d^n) = n \phi(d) = 0 \implies \phi(d) = 0
    \]
    because we are in $\Q$.
    
    For surjectivity, let $\psi: \Delta \to \Q$ be in the fixed subspace.
    Since $\Q$ is abelian, we have to show that there exists a lift
    \[
    \begin{tikzcd}
    \frac{\Delta}{[\Delta, \Delta]} \arrow[r,"f"] \arrow[d,"\psi"]& \frac{D}{[D,D]} \arrow[dl, dashed]
    \\
    \Q &
    \end{tikzcd}
    \]
    By injectivity of $\Q$ such a lift would exist if $f$ were injective but
    \[
    \ker f = \frac{[D,D] \cap \Delta}{[\Delta, \Delta]}.
    \]
    Therefore it suffices to show that $\psi$ vanishes on $[D,D] \cap \Delta$.
    So let $[d_1, d_2] \in \Delta$ for some $d_1, d_2 \in D$.
    Pick $k$ large enough so that $d_2^k \in \Delta$.
    Write
    \[
    d_1 d_2^k d_1^{-1} = ([d_1,d_2] d_2)^k.
    \]
    Since $[d_1, d_2] \in \Delta$ we have that
    \[
    [d_1,d_2] d_2 = d_2 [d_1,d_2] \mod [D,\Delta] \implies ([d_1,d_2] d_2)^k = [d_1,d_2]^k d_2^k \mod [D,\Delta].
    \]
    Since $[D,\Delta] \subseteq \ker \psi$ by invariance under the action, we arrive at
    \[
    0 = \psi(d_1 d_2^k d_1^{-1} d_2^{-k}) = \psi(([d_1,d_2] d_2)^k d_2^{-k}) = \psi([d_1,d_2]^k) = k \psi([d_1,d_2]).
    \]
    It follows that $b_1(\Delta,\Q) = \dim H^1(\Delta,\Q) \geq \dim H^1(D,\Q) = b_1(D,\Q)$.
\end{proof}

\begin{lemma} \label{lem:higher}
For all integers $n>2$ and $m \geq 0$ there are $n$-dimensional manifolds $X$ with $b_1(X) \geq m$ and $\pi_0(Diff(X)) = 0$.
\end{lemma}
\begin{proof}
Theorem 1.1 in \cite{hyperbolicmanifolds} states that there exist infinitely many hyperbolic $n$-manifolds with isometry group being an arbitrary finite group $G$ - for us the case of interest is with $G$ being the trivial group.
By Mostov rigidity, every diffeomorphism is homotopic to an isometry, so these have trivial mapping class group.
Therefore it suffices to show that their construction allows for manifolds with arbitrarily large $H^1(X;\Q)$.

We start with any choice of groups $M \lhd \Delta \lhd \Gamma, D \leq \Gamma$ satisfying the assumptions of \cite[proposition 4.1]{hyperbolicmanifolds}, which in particular means $\Delta /M$ is a nonabelian free group.
The first claim is that we can assume $\Delta/M$ has rank at least $m$.  Let $F \leq \Delta/M$ be a finite index subgroup which is a free group of rank at least $m$.  Let $\Delta' \subseteq \Gamma$ be the preimage of $F$ under the projection $\Delta \to \Delta/M$ and let $\Delta''$ be the largest subgroup that is normal in $\Gamma$ and contained in $\Delta'$.  So we have $\Delta'/M$ a finitely generated free group of rank at least $m$.  Let $M'' = \Delta'' \cap M$.  We claim that $M'' \lhd \Delta'' \lhd \Gamma$ again satisfy the desired hypotheses. Note that $[\Delta : \Delta'] < \infty$ because $F$ is finite index in $\Delta/M$.  Since normal cores of finite index subgroups are again finite index and $[\Gamma : \Delta] < \infty$, we have that $[\Gamma : \Delta''] < \infty$.  Since $M$ is a normal subgroup of $\Delta'$, we have that $M''$ is a normal subgroup in $\Delta''$.  Now the map $\Delta'' \to \Delta' \to \Delta'/M$ factors through and yields an inclusion $\Delta''/M'' \to \Delta'/M$.  Further, since $[\Delta':\Delta''] < \infty$, also $[\Delta'/M : \Delta''/M'']<\infty$ which implies that $\Delta''/M$ is free of rank greater than or equal to $m$ (since finite-index subgroups of a free group $F_n$ are free of rank at least $n$).  Thus by replacing $\Delta$ and $M$ with $\Delta''$ and $M''$ if necessary, we may assume that $\Delta/M$ has rank at least $m$.



With an appropriate choice of $\Gamma$ as a maximal cocompact nonarithmetic lattice in the isometry group of $\mathbb{H}^n$, there will also exist the appropriate $\Delta, D, B, M$ inside of $\Gamma$ (see section 4 of \cite{hyperbolicmanifolds} with the suitable modification as in the previous paragraph so that $\Delta/M$ has rank at least $m$) so that we have a hyperbolic manifold $X := \mathbb{H}^n / B$ with $\pi_0(\operatorname{Diff}(X)) = N_\Gamma(B)/B$.  It suffices to show that $b_1(X) \geq m$.

Note that $X$ is a $K(B,1)$ and so $b_1(X) = \dim H^1(B,\Q)$.
It follows from Lemma \ref{lemma} that
\[
\dim H^1(B,\Q) \geq \dim H^1(D,\Q) = \dim H^1(\Delta,\Q)^{D/\Delta}.
\]
By assumption \cite[proposition 4.1(iii)]{hyperbolicmanifolds} the conjugation action of $D$ is inner on $\Delta/M$.
Therefore 
\[
\dim H^1(\Delta,\Q)^{D/\Delta} \geq \dim H^1(\Delta /M, \Q) \geq m
\]
which implies that $b_1(X) \geq m$, as desired.  
\end{proof}

\begin{theorem}
For all $n >2$, there exists $n$-manifolds $X^n$ with two $\Pin^-$-structures $p_1$ and $p_2$ that are cobordant but not $\Pin^-$-diffeomorphic.  The same result also holds for $\Pin^+$-structures.  
\end{theorem}

\begin{proof}
If a manifold $X^n$ has trivial mapping class group, then any two distinct $\Pin^{\pm}$-structures on $X$ are not diffeomorphic.  
If $b_1(X) > |\Omega_n^{\Pin^\pm}|$ then there must, by the pigeonhole principle, exist a pair of distinct $\Pin$-structures on $X$ that are not cobordant.  
Thus, by using Lemma \ref{lem:higher} to choose $X$ with $b_1(X)$ arbitrarily large and the mapping class group of $X$ trivial, it suffices to prove that $\Omega_n^{\Pin^{\pm}}$ is always finite.  
To see this, note that the groups $\Omega_n^{\Pin^\pm}$ are finitely generated because they are connected to the groups $\Omega^{\Spin}_n$ by a twisted Serre spectral sequence. 
More precisely, there are fibrations of connected spaces
\[
B\Spin \to B\Pin^{\pm} \to B \Z_2
\]
compatible with the structure group maps $B\Spin \to BO$ and $B\Pin^{\pm} \to BO$.
Therefore there is a spectral sequence
\[
H_p(B \Z_2, \Omega^{\Spin}_q) \implies \Omega^{\Pin^{\pm}}_{p+q}
\]
where the coefficients $\Omega^{\Spin}_q$ are twisted.
Now $\Omega^{\Pin^{\pm}}_n$ are finitely generated, since $\Omega^{\Spin}_n$ are finitely generated \cite[page 336]{RS} and finitely generated groups form a Serre class.  
Then by \cite{kirbytaylorcalculation}, it follows that all elements in $\Omega_n^{\Pin^\pm}$ are torsion and therefore the groups $\Omega_n^{\Pin^\pm}$ are all finite.
\end{proof}

The above argument does not work for $\Spin$ structures (since $\Omega_n^{\Spin}$ can be infinite), and although we believe that the analog of Corollary \ref{cor} for $\Spin$-structures does not hold in dimensions greater than 2, we do not know how to prove this.      

\bibliography{pin}

\begin{thebibliography}{Joh80b}

\bibitem[Arf41]{arf}
Cahit Arf.
\newblock {Untersuchungen {\"u}ber quadratische Formen in K{\"o}rpern der
  Charakteristik 2. (Teil I.).}
\newblock {\em Journal f{\"u}r die reine und angewandte Mathematik},
  1941(183):148--167, 1941.

\bibitem[BL05]{hyperbolicmanifolds}
Mikhail Belolipetsky and Alexander Lubotzky.
\newblock Finite groups and hyperbolic manifolds.
\newblock {\em Inventiones mathematicae}, 162(3):459--472, 2005.

\bibitem[DF97]{degtyarev_finashin}
Alexander Degtyarev and Serge Finashin.
\newblock Pin-structures on surfaces and quadratic forms.
\newblock {\em Turkish Journal of Mathematics}, 21(2):187--193, 1997.

\bibitem[DG18]{debray2018arf}
Arun Debray and Sam Gunningham.
\newblock {The Arf-Brown TQFT of Pin-surfaces}.
\newblock {\em Topology and Quantum Theory in Interacti}, 2018.

\bibitem[FM11]{farb_margalit}
Benson Farb and Dan Margalit.
\newblock {\em A primer on mapping class groups}, volume~49.
\newblock Princeton University Press, 2011.

\bibitem[Joh80a]{johnson_bc}
Dennis Johnson.
\newblock Quadratic forms and the {Birman-Craggs} homomorphisms.
\newblock {\em Transactions of the American Mathematical Society},
  261(1):235--254, 1980.

\bibitem[Joh80b]{johnson_quadratic}
Dennis Johnson.
\newblock Spin structures and quadratic forms on surfaces.
\newblock {\em Journal of the London Mathematical Society}, 2(2):365--373,
  1980.

\bibitem[KM04]{kirby_melvin}
Robion Kirby and Paul Melvin.
\newblock Local surgery formulas for quantum invariants and the {Arf}
  invariant.
\newblock {\em Geometry \& Topology Monographs}, 7:213--233, 2004.

\bibitem[KT90a]{kirbytaylor}
Robion~C. Kirby and Laurence~R. Taylor.
\newblock Pin structures on low-dimensional manifolds geometry of
  low-dimensional manifolds, 2 ({D}urham, 1989)({London Mathematical Society
  Lecture Note Series} vol. 151), 1990.

\bibitem[KT90b]{kirbytaylorcalculation}
Robion~C. Kirby and Lawrence~R. Taylor.
\newblock {A calculation of $\Pin^+$ bordism groups}.
\newblock {\em Commentarii Mathematici Helvetici}, 65(1):434--447, 1990.

\bibitem[Sav11]{saveliev_book}
Nikolai Saveliev.
\newblock {\em Lectures on the topology of 3-manifolds: an introduction to the
  Casson invariant}.
\newblock Walter de Gruyter, 2011.

\bibitem[Sto68]{RS}
Robert~E. Stong.
\newblock {\em Notes on cobordism theory}.
\newblock Mathematical notes. Princeton University Press, Princeton, N.J.;
  University of Tokyo Press, Tokyo, 1968.

\end{thebibliography}
\bibliographystyle{alpha}

\end{document}